\documentclass[12pt, A4]{article}
\usepackage{amsmath,amssymb,amsfonts,mathrsfs,hyperref,microtype, amsthm}

\usepackage{geometry}
\usepackage{eufrak}
\usepackage{graphicx,color}
\usepackage[T1]{fontenc}
\usepackage{rotating}
\usepackage{booktabs}
\usepackage{mathtools}
\usepackage{float}
\usepackage{breqn}
\usepackage{graphicx}
\usepackage{caption}

\newtheorem{thm}{Theorem}[section]
\newtheorem{rem}{Remark}[section]
\newtheorem{definition}{Definition}[section]
\newtheorem{lemme}{Lemma}[section]
\newtheorem{prop}{Proposition}[section]

\numberwithin{equation}{section}

\def\EE{\widehat{\eta}}
\def\HH{ \EuFrak H}
\def\N{{\rm I\kern-0.16em N}}
\def\R{{\rm I\kern-0.16em R}}
\def \E{{\rm I\kern-0.16em E}}
\def\P{{\rm I\kern-0.16em P}}
\def\F{{\rm I\kern-0.16em F}}
\def\B{{\rm I\kern-0.16em B}}
\def\C{{\rm I\kern-0.46em C}}
\def\G{{\rm I\kern-0.50em G}}

\newcommand{\ud}{\mathrm{d}}

\numberwithin{equation}{section}

\font\eka=cmex10
\usepackage{ae}

\def\ind{\mathrel{\hbox{\rlap{%
\hbox to 7.5pt{\hrulefill}}\raise6.6pt\hbox{\eka\char'167}}}}
\begin{document}

\title{\Large{\bf Optimal Berry-Esseen bounds  \\ on the Poisson space}}



\renewcommand{\thefootnote}{\fnsymbol{footnote}}

\author{Ehsan Azmoodeh \\ \small{Unit\'{e} de Recherche en Math\'{e}matiques, Luxembourg University} \\ \small{{ \tt ehsan.azmoodeh@uni.lu}} \medskip \\
Giovanni Peccati \\ \small{Unit\'{e} de Recherche en Math\'{e}matiques, Luxembourg University} \\ \small{ {\tt giovanni.peccati@gmail.com}}}

\footnotetext{ Azmoodeh is supported by research project F1R-MTH-PUL-12PAMP.}
\maketitle

\abstract We establish new lower bounds for the normal approximation in the Wasserstein distance of random variables that are functionals of a Poisson measure. Our results generalize previous findings by 
Nourdin and Peccati (2012, 2015) and Bierm\'e, Bonami, Nourdin and Peccati (2013), involving random variables living on a Gaussian space. Applications are given to optimal Berry-Esseen bounds for edge 
counting in random geometric graphs. 


\section{Introduction} 

\subsection{Overview}

Let $(Z,\mathcal{Z})$ be a Borel space endowed with a $\sigma$-finite non-atomic measure $\mu$, and let $\EE$ be a compensated Poisson random measure on the state space $(Z,\mathcal{Z})$, with non-atomic and $\sigma$-finite control measure 
$\mu$ (for the rest of the paper, we assume that all random objects are defined on a common probability space $(\Omega, \mathscr{F}, \P)$). Consider a sequence of centered random variables $F_n = F_n(\EE)$, $n\geq 1$ and assume that, as $ n\to \infty$, ${\rm Var}(F_n)\to 1$  and $F_n$ converges in distribution to a standard Gaussian random variable. In recent years (see e.g. \cite{bp, e-t-kolbound, rey-p-1, rey-p-2, lps, p-s-t-u, r-s, s}) several new techniques  -- based on the interaction between Stein's method \cite{c-g-s-book} and Malliavin calculus \cite{l} -- have been introduced, allowing one to find explicit Berry-Esseen bounds of the type
\begin{equation}\label{e:i}
d(F_n, N)\leq \varphi(n),\quad n\geq 1,
\end{equation}
where $d$ is some appropriate distance between the laws of $F_n$ and $N$, and $\{\varphi(n) :n \geq 1\}$ is an explicit and strictly positive numerical sequence converging to 0. The aim of this paper is to find some general sufficient conditions, ensuring that the rate of convergence induced by $\varphi(n)$ in \eqref{e:i} is {\it optimal}, whenever $d$ equals the 1-Wasserstein distance $d_W$, that is:
\begin{equation}\label{Wass}
d(F_n, N) = d_W(F_n,N): = \sup_{h\in {\rm Lip}(1)} | \mathbb{E}[h(F_n)]-
\mathbb{E}[h(N)] |,
\end{equation}
with ${\rm Lip}(a)$ indicating the set of $a$-Lipschitz mappings on $\R$ ($a>0$). As usual, the rate of convergence induced by  $\varphi(n)$ is said to be optimal if there exists a constant $c\in (0,1)$ (independent of $n$) such that, for $n$ large enough,
\begin{equation}\label{e:opti}
\frac{d_W(F_n,N)}{\varphi(n)}\in (c, 1].
\end{equation}

As demonstrated below, our findings generalize to the framework of random point measures some previous findings (see \cite{bbnp, simone, n-p-exact, n-p-pams}) for random variables living on a Gaussian space. Several important differences between the Poisson and the Gaussian settings will be highlighted as our analysis unfolds. Important new applications $U$-statistics, in particular to edge-counting in random geometric graphs, are discussed in Section \ref{S:applications}.

\subsection{Main abstract result (and some preliminaries)}\label{S:main-resut}

Let the above assumptions and notation prevail. The following elements are needed for the subsequent discussion, and will be formally introduced and discussed in Section \ref{SS : 4Operators}:
\begin{itemize}

\item[--] For every $z\in Z$ and any functional $F = F(\EE)$, the {\it difference} (or {\it add-one cost operator}) $D_zF(\EE) = F(\EE+\delta_z)-F(\EE)$. For reasons that are clarified below, we shall write $F\in {\rm dom} \, D$, whenever $\E\int_Z (D_s F)^2 \mu({\rm d} s)<\infty$.

\item[--] The symbol $L^{-1}$ denotes the {\it pseudo-inverse of the generator of the Ornstein-Uhlenbeck semigroup} on the Poisson space.

\end{itemize}

We also denote by $N\sim \mathscr{N}(0,1)$ a standard Gaussian random variable with mean zero and variance one. It will be also necessary to consider the family
$$
\mathscr{F}_{W}:= \{ f:\R \to \R \, : \, \Vert f' \Vert_\infty \le 1 \, \text{and} \, f'\in {\rm Lip}(2) \},
$$
whereas the notation $\mathscr{F}_0$ indicates the subset of $\mathscr{F}_{W}$ that is composed of twice continuously differentiable functions such that $\Vert f' \Vert_\infty \le 1$ and $\Vert f'' \Vert_\infty \le 2$. 

\medskip

For any two sequences $\{a_n\}_{n \ge 1}$ and $\{b_n\}_{n \ge 1}$ of non-negative 
real numbers, the notation $a_n \sim b_n$ indicates that $\lim_{n \to \infty} \frac{a_n}{b_n}=1$. 

\medskip

The next theorem is the main theoretical achievement of the present paper.

\begin{thm}\label{thm:Main1}
Let $\{F_n: n \ge 1\}$ be a sequence of square-integrable functionals of $\EE$, such that $\E(F_n) =0$, and $F_n \in {\rm dom}\, D$. Let $\{\varphi(n):n\geq 1\}$ be a numerical sequence such that  $\varphi(n)\geq  \varphi_1(n) + \varphi_2(n)$, where 

\begin{align}\label{eq:rates}
\varphi_1(n) & := \sqrt{\E \left( 1 -  \langle DF_n,-DL^{-1}F_n \rangle_{L^2(\mu)} \right)^2 },\\
\varphi_2(n) & : =  \E \int_Z  (D_z F_n)^2 \times \vert D_zL^{-1} F_n \vert \mu(\ud z).
\end{align}

\noindent{\rm \bf (I)} For every $n$, one has the estimate $d_W(F_n, N)\leq \varphi(n)$.

\medskip

\noindent{\rm \bf (II)} Fix $f \in \mathscr{F}_0$, set $R^f_n(z):= \int_0^1 f''(F_n + (1-u) D_z F_n) u \, \ud u$ for any $z \in Z$, and assume moreover that the following asymptotic conditions are in order :

\begin{enumerate}
 \item[\rm (i)] {\rm (a)} $ \varphi(n)$ is finite for every n; {\rm (b)} $\varphi(n) \to 0 $ as $n \to \infty$; and {\rm (c)} there exists $m \ge 1$ such that $\varphi(n) > 0$ for all $n \ge m$. 
 \item[\rm (ii)] For $\mu({\rm d} z)$-almost every $z \in Z$, the sequence $D_z F_n$ converges in probability towards zero.
 \item[\rm (iii)] There exist a centered two dimensional Gaussian random vector $(N_1,N_2)$ with $\E (N^2_1)=\E(N^2_2)=1$, and $\E (N_1 \times N_2)= \rho$, and moreover a real number $\alpha \ge 0$ such that
 
 \begin{equation*}
 \left( F_n , \frac{ 1 - \langle DF_n,-DL^{-1}F_n \rangle_{L^2(\mu)}}{\varphi(n)} \right) \stackrel{\text{law}}{\to} (N_1, \alpha  N_2).
 \end{equation*}
\item[\rm (iv)] There exists a sequence $\{u_n: n \ge 1\}$ of deterministic and non-negative measurable functions such that $ \int_Z u_n(z) \mu(\ud z) / \varphi(n) \to \beta <\infty$, and moreover

\begin{equation*}
\frac{1}{\varphi(n)}\left\{ \int_Z (D_z F_n)^2 \times (-D_zL^{-1} F_n) \times R^f_n(z) \mu(\ud z) - \int_Z u_n(z) \times R^f_n(z) \mu (\ud z)\right\} \stackrel{L^1(\Omega)}{\longrightarrow} 0,
\end{equation*}
and $\sup_n \varphi(n)^{-(1+\epsilon)} \int_Z u_n(z)^{1+\epsilon} \mu(\ud z)<\infty$, for some $\epsilon>0$.

\end{enumerate}
Then, as $n \to \infty$, we have 

\begin{equation*}
\frac{\E \left( f'(F_n)  - F_n  f(F_n) \right) }{\varphi(n)} \to \left\{ \frac{\beta}{2} +  \rho\, \alpha \right\} \, \E (f''(N)).
\end{equation*}

\medskip

\noindent{\bf (III)} If Assumptions {\rm (i)}--{\rm (iv)} at Point {\bf (II)} are verified and $\rho\alpha\neq \frac{\beta}{2}$, then the rate of convergence induced by $\varphi(n)$ is optimal, in the sense of \eqref{e:opti}.

\end{thm}

\begin{rem} { \rm
It is interesting to observe that Assumptions {\bf (II)}-(ii) and {\bf (II)}-(iv) in the statement of Theorem \ref{thm:Main1} do not have any counterpart in the results on Wiener space obtained in \cite{n-p-exact}. To see this, let $X$ denote a isonormal Gaussian process over a real separable Hilbert space $\HH$, and assume that $\{F_n : n\geq 1\}$ is a sequence of smooth functionals (in the 
 sense of Malliavin differentiability) of $X$ --- for example, each element $F_n$ 
 is a finite sum of multiple Wiener integrals. Assume that $\E(F^2_n)=1$, and write 
 $$\varphi(n):= \sqrt{\E ( 1-  \langle D F_n, - D L^{-1}F_n \rangle_{\HH} )^2}.$$ Assume that $\varphi(n)> 0$ for all $n$ and also that, as $n \to \infty$, $\varphi(n) \to 0$ and 
 the two dimensional random vector $\left( F_n, \frac{1 - \langle D F_n, -DL^{-1} F_n \rangle_{\HH}}{\varphi(n)} \right)$ converges in distribution to a centered two dimensional Gaussian vector $(N_1,N_2)$, such that 
 $\E(N^2_1) =\E(N^2_2)=1$ and $\E(N_1 N_2)= \rho \neq 0$. Then, the results of \cite{n-p-exact} imply that, for any function $f \in \mathscr{F}_W$, 
  \begin{equation}\label{optimal-Wiener}
 \frac{\E \left( f'({F}_n)  - {F}_n  f({F}_n) \right) }{\varphi(n)} \to \rho \E (f''(N)),
 \end{equation}
where, as before, $N \sim \mathscr{N}(0,1)$. This implies in particular that the sequence $\varphi(n)$ determines an optimal rate of convergence, in the sense of \eqref{e:opti}.  Also, on a Gaussian space one has that relation $(\ref{optimal-Wiener})$ extends to functions of the type $f_x$, where $f_x$ is the solution of the Stein's equation associated with the indicator 
function $\textbf{1}_{ \{ \cdot \le x\}}$ (see Section \ref{ss:stein} below): in this case the limiting value equals $ \frac{\rho}{3} (x^2 -1) \frac{e^{- \frac{x^2}{2}}}{\sqrt{2 \pi}}$.  
 }
\end{rem}

%
%
%
%
%
%


\section{Preliminaries} \label{S:normal-app}
\subsection{Poisson measures and chaos} \label{SS : PoissMeas} \setcounter{equation}{0}
As before, $\left( Z,\mathcal{Z},\mu \right) $ indicates
a Borel measure space such that $Z$ is a Borel space and $\mu $ is a
$\sigma$-finite and non-atomic Borel measure. We define the
class $\mathcal{Z}_{\mu }$ as $\mathcal{Z}_{\mu }=\left\{ B\in \mathcal{Z}%
:\mu \left( B\right) <\infty \right\} $. The symbol $\EE=\{\EE \left( B\right) :B\in \mathcal{Z}_{\mu }\}$ indicates a \textsl{compensated
Poisson random measure} on $\left( Z,\mathcal{Z}\right) $ with control $\mu
. $ This means that $\EE$ is a collection of random variables
defined on the probability space $\left( \Omega ,\mathscr{F},\P%
\right) $, indexed by the elements of $\mathcal{Z}_{\mu }$, and such that:
(i) for every $B,C\in \mathcal{Z}_{\mu }$ such that $B\cap C=\varnothing $, $%
\EE\left( B\right) $ and $\EE\left( C\right) $ are
independent, (ii) for every $B\in \mathcal{Z}_{\mu }$,$\EE(B)$ has a centered Poisson distribution with parameter $\mu \left( B\right) $. Note that properties
(i)-(ii) imply, in particular, that $\EE$ is an
\textit{independently scattered }(or \textit{completely random})
measure. Without loss of generality, we may assume that $\mathscr{F} = \sigma(\EE)$, and write $L^2(\P) :=L^2(\Omega, \mathscr{F}, \P)$. See e.g. \cite{l, pt-book} for details on the notions evoked above.

\medskip

Fix $n\geq 1.$ We denote by $L^{2}\left( \mu ^{n}\right) $ the
space of real valued functions on $Z^{n}$ that are
square-integrable with respect to $\mu ^{n}$, and we write
$L_{s}^{2}\left( \mu ^{n}\right) $ to indicate the subspace of
$L^{2}\left( \mu ^{n}\right) $ composed of symmetric functions. We also write $L^2(\mu) = L^2(\mu^1) =L_s^2(\mu^1)$. For every $f\in L_{s}^{2}(\mu ^{n})$, we denote by $I_{n}(f)$ the
\textsl{multiple Wiener-It\^{o} integral} of order $n$, of $f$ with
respect to $\EE$. Observe that, for every $m,n\geq 1$, $f\in
L_{s}^{2}(\mu ^{n})$ and $g\in L_{s}^{2}(\mu ^{m})$, one has the
isometric formula (see e.g. \cite{pt-book}):
\begin{equation}
\E\big[I_{n}(f)I_{m}(g)\big]=n!\langle f,g%
\rangle _{L^{2}(\mu ^{n})}\mathbf{1}_{(n=m)}.  \label{iso}
\end{equation}%
The Hilbert space of random variables of the type $I_n(f)$, where
$n\geq 1$ and $f\in L^2_s(\mu^n)$ is called the $n$th \textsl{Wiener
chaos}
associated with $\EE$. We also use the following standard notation: $I_{1}\left( f\right) =%
\EE\left( f\right) $, $f\in L^{2}\left( \mu \right) $;
$I_{0}\left( c\right) =c$, $c\in \mathbb{R}$. The following
proposition, whose content is known as the \textsl{chaotic
representation property} of $\EE$, is one of the crucial
results used in this paper. See e.g. \cite{pt-book}.
\begin{prop}[Chaotic decomposition] \label{P : Chaos}
Every random variable $F\in
L^2(\mathscr{F},\mathbb{P})=L^2(\mathbb{P})$ admits a (unique)
chaotic decomposition of the type
\begin{equation}\label{chaosexpansion}
F = \mathbb{E}(F) + \sum_{n\geq 1}^{\infty} I_n(f_n),
\end{equation}
where the series converges in $L^2$ and, for each $n\geq 1$, the
kernel $f_n$ is an element of $L_{s}^{2}(\mu ^{n})$.
\end{prop}

\subsection{Malliavin operators} \label{SS : 4Operators}
We recall that the space
$L^2(\P;L^2(\mu)) \simeq L^2(\Omega\times Z,
\mathscr{F}\otimes \mathcal{Z}, \P\otimes \mu)$ is the space of the
measurable random functions $u : \Omega\times Z \rightarrow \R $
such that
$$
\mathbb{E}\left[\int_Z u^2_z \, \mu(\ud z)\right] <\infty.
$$
In what follows, given $f\in L^2_s(\mu^q)$ ($q\geq 2$) and $z\in Z$,
we write $f(z,\cdot)$ to indicate the function on $Z^{q-1}$ given by
$(z_1,...,z_{q-1})\rightarrow f(z,z_1,...,z_{q-1})$.
\medskip

\noindent (a) \textit{The derivative operator $D$}. The
derivative operator, denoted by $D$, transforms random variables
into random functions. Formally, the domain of $D$, written ${\rm
dom} D$, is the set of those random variables $F \in
L^2(\mathbb{P})$ admitting a chaotic decomposition
(\ref{chaosexpansion}) such that
\begin{equation}\label{condDer}
\sum_{n\geq 1}n n! \|f_n\|_{L^2(\mu^n)}^2 < \infty.
\end{equation}
If $F$ verifies (\ref{condDer}) (that is, if $F \in {\rm dom}D$),
then the random function  $z\rightarrow D_zF$ is given by
\begin{equation}\label{TheDerivative}
D_z F = \sum_{n \geq 1}n I_{n-1}(f(z,\cdot)), \,\,\, z\in Z.
\end{equation}

%
%
%
%
%
%
%
%
%
%
%
%
%
%

Plainly $DF \in L^2(\P;L^2(\mu))$, for every $F\in {\rm dom}\, D$. For every random variable of the form $F = F(\EE)$ and for every $z\in Z$, we write $F_z =F_z(\EE) = F(\EE+\delta_z)$. The following fundamental result combines classic findings from \cite{lp} and \cite[Lemma 3.1]{pth}.

\begin{lemme}\label{L : DifferenceDerivative}
For every $F\in L^2(\P)$ one has that $F$ is in ${\rm dom}\, D$ if and only if the mapping $z\mapsto (F_z-F)$ is an element of $L^2(\mathbb{P};L^2(\mu))$. Moreover, in this case one has that $D_zF = F_z-F$ almost everywhere ${\rm d}\mu\otimes {\rm d}\P$.
\end{lemme}

\noindent (b) \textit{The Ornstein-Uhlenbeck generator $L$}. The
domain of the {\it Ornstein-Uhlenbeck generator} (see e.g. \cite{l, lp})), written ${\rm dom}\, L$, is given by those $F \in
L^2(\mathbb{P})$ such that their chaotic expansion
(\ref{chaosexpansion}) verifies
$$
\sum_{n\geq 1}n^2 n! \|f_n\|^2_{L^2(\mu^{n})}<\infty.
$$
If $F \in {\rm dom}L $, then the random variable $LF$ is given by
\begin{equation}\label{LF}
LF = -\sum_{n\geq 1} nI_n(f_n).
\end{equation}

We will also write $L^{-1}$ to denote the pseudo-inverse of $L$. Note that $\mathbb{E}(LF)=0$, by definition. The following result
is a direct consequence of the definitions of $D$ and
$L$, and involves the adjoint $\delta$ of $D$ (with respect to the space $L^2(\P ;L^2(\mu))$ --- see e.g. \cite{l, lp} for a proof).
\begin{lemme}\label{L : deltaDL} For every $F \in {\rm dom}L$,
one has that $F \in {\rm dom}D$ and $D F$ belongs to the domain to the adjoint $\delta$ of $D$.
Moreover,
\begin{equation}\label{DdL}
\delta D F = -LF.
\end{equation}
\end{lemme}

\subsection{Some estimates based on Stein's method}\label{ss:stein}

We shall now present some estimates based on the use of Stein's method for the one-dimensional normal approximation. We refer the reader to the two monographs \cite{c-g-s-book, n-p-book} for a detailed presentation of the subject. Let $F$ be a random variable and let $N\sim \mathscr{N}(0,1)$, and consider a real-valued function
$h:\R\rightarrow\R$ such that the expectation $\mathbb{E}[h(X)]$
is well-defined. We recall that the \textsl{Stein equation} associated with $h$
and $F$ is classically given by
\begin{equation}\label{SteinGaussEq}
h(u)-\E[h(F)] = f'(u)-uf(u), \quad u\in\R.
\end{equation}
A solution to (\ref{SteinGaussEq}) is a function $f$ depending on
$h$ which is Lebesgue a.e.-differentiable, and such that there
exists a version of $f'$ verifying (\ref{SteinGaussEq}) for every
$x\in\R$. The following lemma gathers together some fundamental relations. Recall the notation $\mathscr{F}_W$ and $\mathscr{F}_0$ introduced in Section \ref{S:main-resut}.

\begin{lemme}\label{Stein_Lemma_Gauss}
\begin{enumerate}

\item[\rm (i)] If $h\in {\rm Lip}(1)$, then (\ref{SteinGaussEq}) has a solution $f_h$ that is an element of $\mathscr{F}_W$.

\item[\rm (ii)] If $h$ is twice continuously differentiable and $\|h'\|_\infty, \|h''\|_\infty\leq 1$, then (\ref{SteinGaussEq}) has a solution $f_h$ that is an element of $\mathscr{F}_0$.

\item[\rm (iii)] Let $F$ be an integrable random variable. Then,
$$
d_W(F,N) = \sup_{h\in {\rm Lip(1)} }\left| \E[f_h(F)F - f'_h(F)]\right| \leq \sup_{f\in \mathscr{F}_W } \left| \E[f(F)F - f'(F)]\right|.
$$

\item[\rm (iv)] If, in addition, $F$ is a centered element of ${\rm dom}\, D$, then
\begin{eqnarray}\label{e:pstu}
d_W(F,N)&\leq& \sqrt{\E \left( 1 -  \langle DF,-DL^{-1}F \rangle_{L^2(\mu)} \right)^2 }\\
&&\hspace{3.3cm}+\E \int_Z  (D_z F)^2 \times \vert D_zL^{-1} F \vert \mu(\ud z).\notag
\end{eqnarray}

%
%
%

\end{enumerate}

\end{lemme}

Both estimates at Point (i) and (ii) follow e.g. from \cite[Theorem 1.1]{d}. Point (iii) is an immediate consequence of the definition of $d_W$, as well as of the content of Point (i) and Point (ii) in the statement. Finally, Point (iv) corresponds to the main estimate established in \cite{p-s-t-u}.

\section{Proof of Theorem \ref{thm:Main1}}\label{subs:proof}

We start with a general lemma. 

\begin{lemme}\label{general-lemma1}
 Let $F$ be such that $\E(F)=0$ and $F \in {\rm dom} D$. Assume that $N \sim \mathscr{N}(0,1)$. For any $f \in \mathscr{F}_{ 0}$, and $z \in Z$ we set
 \begin{equation}\label{eq:reminder}
 R^f(z):= \int_{0}^{1} f'' \left( F + (1- u) D_z F \right) u \,  \ud u.
 \end{equation}
Then,
  \begin{equation*}
 \begin{split}
  \E \left( f'(F) - F f(F) \right) & = \E \Big( f'(F) \left( 1 - \langle D F , - D L^{-1} F \rangle_{L^(\mu)} \right) \Big) \\
  & +  \E \int_Z (D_z F)^2 \times (- D_z L^{-1} F) \times R^f(z) \mu (\ud z).
 \end{split}
 \end{equation*}
\end{lemme}

\begin{proof}
Using Lemma \ref{L : deltaDL} and the characterization of $\delta$ as the adjoint of $D$, we deduce that, for any $f \in \mathscr{F}_0$
\begin{equation}\label{compu1}
\begin{split}
\E(F f(F)) &= \E (L L^{-1} F f(F)) = \E (\delta(-DL^{-1} F) f(F))\\
& = \E ( \langle D f(F), - D L^{-1} F \rangle_{L^2(\mu)}).
\end{split}
\end{equation}
In view of Lemma \ref{L : DifferenceDerivative}) and of a standard application of Taylor formula, one immediately infers that
  \begin{equation}\label{compu2}
  \begin{split}
  D_z f(F)& := f(F_z) - f(F) = f'(F) D_z F +  \int_{F}^{F_z} f''(u) (F_z - u) \ud u \\
  &= f'(F) D_z F + (D_z F)^2 \times  \int_{0}^{1} f''(F + (1-u) D_z F) u \, \ud u. 
  \end{split}
 \end{equation}
Plugging $(\ref{compu2})$ into $(\ref{compu1})$, we deduce the desired conclusion.
\end{proof}

\begin{proof}[Proof of Theorem \ref{thm:Main1}] Part {\bf (I)} is a direct consequence of Lemma \ref{Stein_Lemma_Gauss}-(iv). To prove Point {\bf (II)}, we fix $f \in \mathscr{F}_0$, and use Lemma \ref{general-lemma1} to deduce that
\begin{equation*}
 \begin{split}
 \frac{\E \left( f'(F_n) - F_n f(F_n) \right)}{\varphi(n)} & = \E \left( f'(F_n) \times \frac{ 1 - \langle D F_n , - D L^{-1} F_n \rangle_{L^2(\mu)} }{\varphi(n)} \right)\\
 & \hspace{-2cm} +\frac{1}{\varphi(n)} \E \int_Z (D_z F_n)^2 \times (- D_z L^{-1} F_n) \times R_n^f(z) \mu (\ud z)\\
 &: = I_{1,n} + I_{2,n}.
 \end{split}
\end{equation*}
Assumption (iii) implies that $$\sup_{n \ge 1} \E \left( \frac{ 1 - \langle D F_n , - D L^{-1} F_n \rangle_{L^2(\mu)} }{\varphi(n)}  \right)^2 < +\infty.$$
Since $\Vert f' \Vert_\infty\le 1$ by assumption, we infer that the class
\begin{equation*}
\left\{ f'(F_n) \times \frac{ 1 - \langle D F_n , - D L^{-1} F_n \rangle_{L^2(\mu)} }{\varphi(n)} : n\geq 1 \right\}\end{equation*}
is uniformly integrable. Assumption (iii) implies therefore that, as $n \to \infty$, 
\begin{equation*}
I_{1,n} \to \E \left( f'(N_1) \times \alpha N_2 \right) = \rho \times \alpha \, \E f''(N).
\end{equation*}
To deal with the term $I_{2,n}$, first note that for each $z \in Z$, Assumptions (ii) and (iii) and Slutsky Theorem imply that, for any $u \in (0,1)$, $$F_n + (1-u) D_z F_n \stackrel{{law}}{\to} N. $$ 
Therefore, using the fact that $\Vert f'' \Vert_\infty\le 2$, and by a direct application of the dominated convergence Theorem, we infer that  

\begin{equation}\label{eq:rate-limit}
\E R^f_n(z) \to \int_0^1 \E (f''(N)) u \ud u = \frac{1}{2}\E f''(N).
\end{equation}
At this point, Assumption (iv) and the triangle inequality immediately imply that, in order to obtain the desired conclusion, it is sufficient to prove that, as $n\to \infty$,
\begin{equation}\label{eq:last-step}
\frac{1}{\varphi(n)} \int_Z u_n(z) \times \left\{ \E R^f_n(z) -   \frac{1}{2} \E f''(N) \right\} \mu (\ud z)  \to 0.
\end{equation}
To show $(\ref{eq:last-step})$, it is enough to prove that the function integrated on the right-hand side is uniformly integrable: this is 
straightforward, since $\vert R^f_n(z) -   \frac{1}{2} \E f''(N) \vert \le 2$, and of the fact that the sequence $n\mapsto \varphi(n)^{-(1+\epsilon)} \int_Z u_n(z)^{1+\epsilon} \mu (\ud z)$ is bounded.
In view of the first equality in Lemma \ref{Stein_Lemma_Gauss}-(iii), to prove the remaining Point {\bf (III)} in the statement, it is enough to show that there exists a function $h$ such that $\|h'\|_\infty ,\,  \|h''\|_\infty \leq 1$, and $\E[f''_h(N)]\neq 0$. Selecting $h(x) = \sin x$, we deduce from \cite[formula (5.2)]{bbnp} that $\E f''_h(N) = 3^{-1}\E[\sin(N)H_3(N)] = e^{-1/2}>0$, thus concluding the proof.
\end{proof}

\section{Applications to $U$-statistics}\label{S:applications}

\subsection{Preliminaries}\label{ss:premU}

In this section, we shall apply our main results to the following situation:
\begin{enumerate}
\item[--] $\EE$ is a compensated Poisson measure on the product space $( \R_+\!\times\! Z, \mathscr{B}(\R_+)\!\otimes\hspace{-1.1mm} \mathcal{Z})$ (where $(Z,\mathcal{Z})$ is a Borel space) with control measure given by 
\begin{equation}\label{e:controlnu}
\nu := \ell\times \mu,
\end{equation} 
with $\ell({\rm d} x) = {\rm d} x$ equal to the Lebesgue measure and $\mu$ equal to a $\sigma$-finite Borel measure with no atoms.

\item[--] For every $n\geq 1$, we set $\EE_n$ to be the Poisson measure on $(Z,\mathcal{Z})$ given by the mapping $A\mapsto \EE_n(A) : = \EE([0,n]\times A)$ ($A\in \mathcal{Z}_\mu$), in such a way that $\EE_n$ is a Poisson measure on $(Z,\mathcal{Z})$, with intensity $\mu_n := n\times \mu$.

\item[--] For every $n$, the random variable $F_n$ is a $U$-statistic of order $2$ with respect to the Poisson measure $\eta_n := \EE_n +\mu_n $, in the sense of the following definition.

\end{enumerate}

\begin{definition}{\rm
A random variable $F$ is called a $U$-{\it statistic} of order $2$, based on the Poisson random measure $\eta_n$ defined above, if there exists a kernel $h \in L^1_s(\mu^2)$ (that is, $h$ is symmetric and in $L^1_s(\mu^2)$, such that
\begin{equation}\label{U-statistics}
 F = \sum_{(x_1,x_2) \in \eta^2_{n, \neq}} h (x_1,x_2),
\end{equation}
where the symbol $\eta^2_{n,\neq}$ indicates the class of all $2$-dimensional vectors $ (x_1,x_2)$ such that $x_i $ is in the support of $\eta_n$ ($i=1,2$) and $x_1 \neq x_2$.}
\end{definition}

We recall that, according to the general results proved in \cite[Lemma 3.5 and Theorem 3.6]{r-s}, one has that, if a random variable $F$ as in \eqref{U-statistics} is square-integrable, then necessarily 
$h\in L^2_s(\mu^2)$, and $F$ admits a representation of the form 
\begin{equation}\label{e:z}
F = \E(F) +  F_1 + F_2 := \E(F) + I_1 (h_1) + I_2 (h_2),
\end{equation}
where $I_1$ and $I_2$ indicate (multiple) Wiener-It\^o integrals of order 1 and 2, respectively, with respect to $\EE$, and 
\begin{eqnarray}\label{kernels}
h_1(t,z)  &:=& 2{\bf 1}_{[0,n]}(t) \int_{Z} \, h(a,z)\, \mu_n({\rm d} a) \\
&=& 2{\bf 1}_{[0,n]}(t) \int_{\R_+\times Z}  {\bf 1}_{[0,n]}(s)h(a,z) \, \nu ({\rm d} s, {\rm d} a)\notag \\ 
&=& 2{\bf 1}_{[0,n]}(t) n \int_{Z} \, h(a,z)\, \mu({\rm d} a)\in L^2(\mu),\notag\\
\quad h_2\big((t_1,x_1), (t_2,x_2)\big)& :=&  {\bf 1}_{[0,n]^2}(t_1,t_2) h(x_1, x_2), \label{kernels2}
\end{eqnarray}
where $\nu$ is defined in \eqref{e:controlnu}.

\subsection{Edge-counting in random geometric graphs}\label{ss:edge}

Let the framework and notation of Section \ref{ss:premU} prevail, set $Z = \R^d$, and assume that $\mu$ is a probability measure that is absolutely continuous with respect to the Lebesgue measure, with a density $f$ that is bounded and everywhere continuous. It is a standard result that, in this case, the non-compensated Poisson measure $\eta_n$ has the same distribution as the point process
$$
A\mapsto \sum_{i=1}^{N_n} \delta_{Y_i}(A),\quad A\in \mathscr{B}(\R^d),
$$ 
where $\delta_y$ indicates the Dirac mass at $y$, $\{Y_i :i\geq 1\}$ is a sequence of i.i.d. random variables with distribution $\mu$, and $N_n$ is an independent Poisson random variable with mean $n$. Throughout this section, we consider a sequence $\{t_n: n\geq 1\}$ of strictly positive numbers decreasing to zero, and consider the sequence of kernels $\{h_n : n\geq 1\}$ given by
$$
h_n : \R^d\times \R^d \to  \R_+ : (x_1,x_2) \mapsto h_n(x_1,x_2) := \frac12 {\bf 1}_{\{0<\|x_1-x_2\|\leq t_n\}},
$$
where, here and for the rest of the section, $\| \cdot \|$ stands for the Euclidean norm in $\R^d$. Then, it is easily seen that, for every $n$, the $U$-statistic 
\begin{equation}\label{edges}
 F_n := \sum_{(x_1,x_2) \in \eta^2_{n, \neq}} h_n (x_1,x_2),
\end{equation}
equals the number of edges in the random geometric graph $(V_n, E_n)$ where the set of vertices $V_n$ is given by the points in the support of $\eta_n$, and $\{x,y\}\in E_n$ if and only if $0<\|x-y\|\leq t_n$ (in particular, no loops are allowed).

We will now state and prove the main achievement of the section, refining several limit theorems for edge-counting one can find in the literature (see e.g. \cite{bp, rey-p-1, rey-p-2, penrosebook, r-s, r-s-t} and the references therein). Observe that, quite remarkably, the conclusion of the forthcoming Theorem \ref{t:maingraph} is independent of the specific form of the density $f$.For every $d$, we denote by $\kappa_d$ the volume of the ball with unit radius in $\R^d$.

\begin{thm}\label{t:maingraph} Assume that $nt_n^d\to \infty$, as $n\to \infty$. 

\begin{enumerate}

\item[\rm (a)] As $n\to\infty$, one has the exact asymptotics
\begin{equation}\label{e:as}
\E(F_n)\sim \frac{\kappa_d\, n^2\, t_n^d}{2}  \int_{\R^d} f(x)^2 {\rm d}x, \quad {\rm Var}(F_n)\sim \frac{\kappa^2_d\, n^3(t_n^d)^{2}}{4}\, \int_{\R^d} f(x)^3 {\rm d}x.
\end{equation}

\item[\rm (b)] Define 
$$
\tilde{F}_n := \frac{F_n - \E F_n}{\sqrt{{\rm Var}(F_n)}}, \quad n\geq 1,
$$ 
and let $N\sim \mathscr{N}(0,1)$. Then, there exists a constant $C\in (0, \infty)$, independent of $n$, such that 
\begin{equation}\label{e:swag}
d_W(\tilde{F}_n, N)\leq \varphi(n):= Cn^{-1/2}.
\end{equation}

\item[\rm (c)] If moreover $ n\, (t_n^d)^3 \to \infty$, then there exists a constant $0<c<C$ such that, for $n$ large enough,
$$
c\, n^{-1/2} \leq d_W(\tilde{F}_n, N) \leq C\,  n^{-1/2},
$$ 
and the rate of convergence induced by the sequence $\varphi(n) = C\, n^{-1/2}$ is therefore optimal.
\end{enumerate}
\end{thm}
\begin{proof} The two asymptotic results at Point (a) follow from \cite[Proposition 3.1]{penrosebook} and \cite[formula (3.23)]{penrosebook}, respectively. Point (b) is a special case of the general estimates proved in \cite[Theorem 3.3]{rey-p-2}. In order to prove Point (c) it is therefore sufficient to show that the sequence $\tilde{F}_n$ verifies Assumptions (ii), (iii) and (iv) of Theorem \ref{thm:Main1}-{\bf (II)}, with respect to the control measure $\nu$ defined in \eqref{e:controlnu}, and with values of $\alpha,\,  \beta$ and $\rho$ such that $\alpha\, \rho \neq \beta/2$. First of all, in view of \eqref{e:z}, one has that, a.e. ${\rm d}t\otimes \mu( {\rm d} z)$,
\begin{equation*}
D_{t,z} \tilde{F}_n = \frac{1}{\sqrt{{\rm Var}(F_n)}} \left\{h_{1,n}(t,z) + 2 I_1 (h_{2,n} (t,z,\cdot)) \right\},
\end{equation*}
where the kernels $h_{1,n}$ and $h_{2,n}$ are obtained from \eqref{kernels} and \eqref{kernels2}, by taking $h=h_n$. Since $h_{1,n}(t,z) =2\, {\bf 1}_{[0,n]}(t)  n \int_{\R^d} h(z,a) \mu({\rm d} a)$, we obtain that $\text{Var}(F_n)^{-\frac{1}{2}} \times h_{1,n}(t,z) = O (( t_n^{2d} n)^{-\frac{1}{2}}) \to 0$, as $n \to \infty$. Also, using the isometric properties of Poisson multiple integrals,
\begin{equation*}
 \E \Big( \frac{I_1 (h_{2,n} ((t,z),\cdot))}{\sqrt{\text{Var}(F_n)}} \Big)^2 \leq  \frac{ n \int_{\R^d} h_n (z,x) \mu (\ud x)}{\text{Var}(F_n)} = O (( nt_n^d)^{-2}) \to 0.
\end{equation*}
It follows that $D_{t,z} \tilde{F}_n$ converges in probability to zero for ${\ud}\nu$-almost every $(t,z)\in \R_+\times \R^d$, and Assumption (ii) of Theorem \ref{thm:Main1}-{\bf (II)} is therefore verified. In order to show that Assumption (iii) in Theorem \ref{thm:Main1}-{\bf (II)} also holds, we need to introduce three (standard) auxiliary kernels:
\begin{eqnarray*}
h_{2,n}\star_2^1 h_{2,n}(t,x) &:=& {\bf 1}_{[0,n]}(t) n \int_{\R^d} h_n^2(x,a) \mu({\rm d} a)\\
h_{2,n}\star_1^1 h_{2,n}((t,x), (s,y)) &:=& {\bf 1}_{[0,n]}(t){\bf 1}_{[0,n]}(s) n \int_{\R^d} h_n(x,a)h_n(y,a) \mu({\rm d} a)\\
h_{1,n}\star_1^1 h_{2,n}(t,x)&:=& 2\,  {\bf 1}_{[0,n]}(t)n^2 \int_{\R^d} \int_{\R^d} h_n(a,y)h_n(x,y) \mu({\rm d} a)\mu({\rm d} x).
\end{eqnarray*}
The following asymptotic relations (for $n\to \infty$) can be routinely deduced from the calculations contained in \cite[Proof of Theorem 3.3]{rey-p-2} (recall that the symbol `$\sim$' indicates an exact asymptotic relation, and observe moreover that the constant $C$ is the same appearing in \eqref{e:swag}):
\begin{eqnarray}
\| h_{2,n} \star_2^1 h_{2,n}\|^2_{L^2(\nu)} &=& O(n^3 (t_n^d)^2) \label{e:a1} \\
\| h_{2,n} \star_1^1 h_{2,n}\|^2_{L^2(\nu^2)} &=& O(n^4 (t_n^d)^3)\label{e:a2}\\
\| h_{1,n} \star_1^1 h_{2,n}\|^2_{L^2(\nu)} &\sim & \frac{\kappa_d^4 n^5(t_n^d)^4}{4}\int_{\R^d} f(x)^5 {\rm d} x \label{e:a3}\\
\langle h_{1,n} , h_{1,n} \star_1^1 h_{2,n}\rangle _{L^2(\nu)} &\sim & \frac{\kappa_d^3 n^4(t_n^d)^3}{2}\int_{\R^d} f(x)^4 {\rm d} x \label{e:a4}\\
\varphi(n){\rm Var}(F_n) &\sim& C\frac{\kappa_d^2 n^{5/2}(t_n^d)^2 }{4} \int_{\R^d}  f(x)^3{\rm d} x\label{e:a5}\\
\| h_{1,n} \star_1^1 h_{2,n}\|^3_{L^3(\nu)} &=& O(n^7 (t_n^d)^6)\label{e:a6}\\
\| h_{1,n}\|^3_{L^3(\nu)} &\sim& \kappa_d^3 n^4 (t_n^d)^3 \int_{\R^d} f(x)^4 \ud x \label{e:a7}\\
\| h_{1,n}\|^4_{L^4(\nu)} &=& O(n^5 (t_n^d)^4)\label{e:a8}\\
\| h_{2,n}\|^2_{L^2(\nu)} &=& O(n^2 (t_n^d))\label{e:a9}.
\end{eqnarray}
Using the fact that, by definition, $L^{-1} Y = -q^{-1} Y $ for every random variable $Y$ living in the $q$th Wiener chaos of $\EE$, we deduce that (using the control measure $\nu$ defined in \eqref{e:controlnu})
\begin{equation*}
 \begin{split}
  \langle D \tilde{F}_n  & , - D L^{-1} \tilde{F}_n \rangle_{L^2(\nu)}  = \frac{1}{ \text{Var}(F_n)} \Big\{ \int_{\R_+\times \R^d} h^2_{1,n}(t,z) \nu (\ud t, \ud z)\\
  & \!\!\!+ 3 \int_{\R_+\times \R^d} \!\!h_{1,n}(t, z) I_1 (h_{2,n} ((t,z),\cdot)) \nu (\ud t, \ud z)  + 2 \int_{\R_+\times \R^d} I^2_1 ( h_{2,n} ((t,z),\cdot)) \nu (\ud t, \ud z) \Big\}.
 \end{split}
\end{equation*}
Using a standard multiplication formula for multiple Poisson integrals (see e.g. \cite[Section 6.5]{pt-book}) on the of the previous equation, one deduces that
\begin{equation*}
 \begin{split}
  &\frac{ \langle D \tilde{F}_n , - D L^{-1}\tilde{F}_n \rangle_{L^2(\nu)} - 1}{\varphi(n)}  \\
  &=\frac{1}{\varphi(n)\text{Var}(F_n)} \Big\{ 3 I_1 ( h_{1,n} \star^1_1 h_{2,n}) + 2 I_1 (h_{2,n} \star^1_2 h_{2,n}) + 2 I_2 (h_{2,n} \star^1_1 h_{2,n}) \Big\} \\
  &:= X_{1,n} + X_{2,n} + X_{3,n}.
 \end{split}
\end{equation*}
Now, in view of \eqref{e:a1}, \eqref{e:a2} and \eqref{e:a5}, one has that, as $n\to\infty$,
\begin{equation*}
 \E (X_{2,n}^2) =  O ((nt_n^d)^{-2}) \to 0, \quad \text{and}\quad  \E (X_{3,n}^2) =  O ((nt_n^d)^{-1})\to 0. 
\end{equation*}
Also, \eqref{e:a3} yields that
\begin{equation*}
 \begin{split}
 \E (X_{1,n}^2) \longrightarrow \frac{9\int_{\R^d} f(x)^5 {\ud x} }{C^2(\int_{\R^d} f(x)^3 {\ud x})^2}: = \alpha^2  >  0.
 \end{split}
\end{equation*}
Finally, in view of \eqref{e:a6} and \eqref{e:a7},
\begin{equation*}
\Big\Vert \frac{h_{1,n} \star^1_1 h_{2,n}}{\varphi(n) \text{Var}(F_n)} \Big\Vert^3_{L^3(\nu)}, \Big\Vert \frac{h_{1,n} }{\varphi(n) \text{Var}(F_n)} \Big\Vert^3_{L^3(\nu)} = O(n^{- \frac{1}{2}} )\to 0, \quad \mbox{and} \and 
\end{equation*}
A standard application of \cite[Corollary 3.4]{p-z-multi} now implies that, as $n\to\infty$
\begin{equation}\label{e:cc}
\Big(\frac{I_1(h_{1,n})}{\sqrt{\text{Var}(F_n)}} , -\frac{I_1(3 h_{1,n}\star^1_1 h_{2,n})}{\varphi(n) \text{Var}(F_n)}\Big) \stackrel{\text{law}}{\to} (Z_1,Z_2),
\end{equation}
where $Z_1 \sim \mathscr{N}(0,1)$ and $Z_2 \sim \mathscr{N}(0,\alpha^2)$ are two jointly Gaussian random variables such that
\begin{equation*}
 \begin{split}
 \rho' := \E(Z_1Z_2) = -\lim_n \frac{3}{\varphi(n) \text{Var}^{\frac{3}{2}}(F_n)} \langle h_{1,n}, h_{1,n}\star^1_1 h_{2,n} \rangle_{L^2(\nu)} = -\frac{12}{C}\frac{\int_{\R^d} f(x)^4 {\rm d} x}{\left( \int_{\R^d} f(x)^3 {\rm d} x\right)^{3/2}}.
  \end{split}
\end{equation*}
Now, since relation \eqref{e:a9} implies that ${\rm Var}(F_n)^{-1/2} I_2(h_{2,n})$ converges to zero in probability, we deduce that the sequence 
$$
\left( \tilde{F}_n, \frac{ 1-\langle D \tilde{F}_n , - D L^{-1}\tilde{F}_n \rangle_{L^2(\nu)} }{\varphi(n)}\right), \quad n\geq 1,
$$
converges necessarily to the same limit as the one appearing on the RHS of \eqref{e:cc}. We therefore conclude that Assumption (iii) in Theorem \ref{thm:Main1}-{\bf (II)} is verified with $\alpha$ defined as above, and $\rho := \rho'/\alpha$. To conclude the proof, we will now show that Assumption (iv) in Theorem \ref{thm:Main1}-{\bf (II)} is satisfied for 
\begin{equation}\label{e:betas}
u_n = h_{1,n}^3\times {\rm Var}(F_n)^{-3/2} \quad \mbox{and}\quad \beta = \frac{8}{C}\frac{\int_{\R^d} f(x)^4 {\rm d} x}{\left( \int_{\R^d} f(x)^3 {\rm d} x\right)^{3/2}}.
\end{equation}
To see this, we use again a product formula for multiple stochastic integrals to infer that
\begin{equation*}
 \begin{split}
  \frac{1}{ \varphi(n)} & \E  \int_{\R_+\times \R^d} (D_{t,z} \tilde{F}_n)^2 \times (- D_{t,z} L^{-1} \tilde{F}_n) \times R^f_n(t,z) \, \nu (\ud t, \ud z)\\
  & =   \frac{ (\text{Var}(F_n))^{- \frac{3}{2}}}{ \varphi(n)} \Big\{ \E \int_{\R_+\times \R^d}  h^3_{1,n}(t, z) \times R^f_n(t,z) \, \nu (\ud t, \ud z)\\
  &+ 5 \E \int_{\R_+\times \R^d} h^2_{1,n}(t,z) \, I_1 (h_{2,n} ((t,z),\cdot)) \times R^f_n(t,z) \, \nu (\ud t, \ud z) \\
  &+ 8 \E \int_{\R_+\times \R^d} h_{1,n}(t,z) \times I_1 (h_{2,n}((t,z),\cdot))^2 \times R^f_n(t, z) \, \nu (\ud t, \ud z) \Big\}\\
  & := B_{1,n}+B_{2,n}+B_{3,n}.
 \end{split}
\end{equation*}
Since relations \eqref{e:a7} and \eqref{e:a8} imply that, as $n\to\infty$ and using the notation \eqref{e:betas},
$$
\int_{\R_+\times \R^d} \frac{u_n(t,z)}{\varphi(n)} \, \nu(\ud t, \ud z)\to \beta, \,\ \mbox{and} \,\, \int_{\R_+\times \R^d} \left (\frac{u_n(t,z)}{\varphi(n)}\right)^{4/3} \, \nu(\ud t, \ud z) = O(n^{-1/3}),
$$
the proof is concluded once we show that $\E B_{2,n}, \,  \E B_{3,n} \to 0$. In order to deal with $B_{2,n}$, we use the fact that $|R_n^f|\leq 1$, together with the Cauchy-Schwarz and Jensen inequalities and the isometric properties of multiple integrals, to deduce that
\begin{eqnarray*}
&&\left|  \E \int_{\R_+\times \R^d} h^2_{1,n}(t,z) \, I_1 (h_{2,n} ((t,z),\cdot)) \times R^f_n(t,z) \, \nu (\ud t, \ud z)\right|\\
&& \leq n^{7/2} \int_{\R^d} \left (\int_{\R^d} h_n(z,a)  \mu(\ud a)\right)^2 \sqrt{ \int_{\R^d} h^2_n(z,a)  \mu(\ud a)  } \, \mu(\ud z)\\
&& \leq n^{7/2} \sqrt{ \int_{\R^d} \left (\int_{\R^d} h_n(z,a)  \mu(\ud a)\right)^2 \left( \int_{\R^d} h^2_n(z,a)  \mu(\ud a) \right) \, \mu(\ud z) } \\
&& = O(n^{7/2}(t_n^d)^{3/2}).
\end{eqnarray*}
Since we work under the assumption that $n(t_n^d)^3\to\infty$, this yields that, as $n\to \infty$, $\E B_{2,n} = O(n^{-1/2}(t_n^d)^{-3/2})\to 0$. Analogous computations yield that $\E B_{3,n} = O((n\, t_n^d)^{-1})\to 0$, and this concludes the proof.
\end{proof}

\subsection{Geometric $U$-statistics of order $2$}\label{subs:U-statistic}
Our approach is robust enough for extending to more general classes of $U$-statistics. In order to see this, let the framework and notation of Section \ref{ss:premU} prevail, and consider a sequence of {\it geometric} $U$-statistics of order $2$ given by
\begin{equation}
\label{eq:Ustat2}
F_n := \sum_{(x_1,x_2) \in \eta^2_{n, \neq}} h (x_1,x_2), \qquad h \in L^1(\mu^2) \cap L^2(\mu^2).
\end{equation}
A proof similar to that of Theorem \ref{t:maingraph} yields the following statement. (Note that items {\rm (a)} and {\rm (b)} in the forthcoming theorem are a consequence of \cite[Theorem 7.3]{rey-p-2}, as well \cite{r-s}).

\begin{thm}\label{thm:Ustat2}
Assume that the kernels $h_{1,n}$ and $h_{2,n}$ are given by the RHS of $(\ref{kernels})$ and $(\ref{kernels2})$.
\begin{enumerate}
\item[\rm (a)] Let $h_1(z):= \int_Z h(x,z) \mu(\ud x)$, for every $z \in Z$. If $\Vert h_1\Vert_{L^2(\mu)} > 0$, then as $n \to \infty$, one has the exact asymptotic 
\begin{equation*}
{\rm Var} (F_n) \sim { \rm Var} (F_{1,n}) \sim \Vert h_1\Vert^2_{L^2(\mu)} \, n^3.
\end{equation*}
\item[\rm (b)] Define 
$$
\widetilde{F}_n := \frac{F_n - \E F_n}{\sqrt{{\rm Var}(F_n)}}, \quad n\geq 1.
$$ 
Set $\tilde{h}_{1,n} = \left({\rm Var}(F_{1,n}) \right)^{- \frac{1}{2}} h_{1,n}$, and 
$\widetilde{\varphi}(n):= \Vert \tilde{h}_{1,n} \Vert^3_{L^3(\mu_n)} = \Vert h_1 \Vert^3_{L^3(\mu)} \times \Vert h_1\Vert^{-3}_{L^2(\mu)} \times n^{- \frac{1}{2}}$. Let $N\sim \mathscr{N}(0,1)$. Then, there exists a 
constant $C\in (0, \infty)$, independent of $n$, such that 
\begin{equation}\label{e:swag}
d_W(\widetilde{F}_n, N)\leq C \, \widetilde{\varphi}(n).
\end{equation}
\item[\rm (c)] Let $h_2(x_1,x_2):= h(x_1,x_2)$, and denote 
\begin{align}\label{variance-corrolation}
\alpha^2_{h_1,h_2}& := \frac{9 \Vert h_1 \Vert^2_{L^2(\mu)} \, \Vert h_1 \star^1_1 h_2 \Vert^2_{L^2(\mu)}}{ \Vert h_1 \Vert^6_{L^3(\mu)}}, \quad \text{and} \\
\rho_{h_1,h_2} & := - \frac{ 3 \langle h_1, h_1 \star^1_1 h_2 \rangle_{L^2(\mu)}}{ \Vert h_1 \Vert^{3}_{L^3(\mu)}  }.
\end{align}
If moreover $h(x_1,x_2)\ge 0$, for $(x_1,x_2)\in Z^2$ a.e. $\mu^2$, and also $\alpha_{h_1,h_2} \times \rho_{h_1,h_2} \neq - 1/2$, then there exists a constant $0<c<C$ such that, for $n$ large enough,
$$
c\, \widetilde{\varphi}(n) \leq d_W(\widetilde{F}_n, N) \leq C\,  \widetilde{\varphi}(n),
$$ 
and the rate of convergence induced by the sequence $\varphi(n):= C\, \widetilde{\varphi}(n)$ is therefore optimal.
\end{enumerate}

\end{thm}

\end{document}